\documentclass[12pt]{amsart}

\usepackage{amssymb}
\usepackage{amsthm}
\usepackage[utf8]{inputenc}
\usepackage{graphicx}
\usepackage[all]{xy}
\usepackage{color}
\usepackage{latexsym}
\usepackage{euscript}
\usepackage{ifthen}
\usepackage{wrapfig}
\usepackage{multicol}
\usepackage{paralist}
\usepackage{verbatim}
\usepackage{algorithmicx}
\usepackage{algpseudocode}
\usepackage{algorithm}
\usepackage{placeins}
\usepackage{caption}
\usepackage{subcaption}
\usepackage{longtable}
\usepackage{lscape}
\usepackage{listings}

\usepackage{hyperref}

\definecolor{darkblue}{RGB}{0,0,160}
\hypersetup{
colorlinks,%
citecolor=black,%
filecolor=black,%
linkcolor=darkblue,%
urlcolor=darkblue
}

\numberwithin{equation}{section}
\theoremstyle{plain}%
\newtheorem{theorem}{Theorem}[section]
\newtheorem{proposition}[theorem]{Proposition}
\newtheorem{lemma}[theorem]{Lemma}
\newtheorem{corollary}[theorem]{Corollary}

\newtheorem{problem}[theorem]{Problem}

\theoremstyle{definition}
\newtheorem{definition}[theorem]{Definition}
\newtheorem{remark}[theorem]{Remark}
\newtheorem{example}[theorem]{Example}

\renewcommand{\>}{\rangle}
\renewcommand{\i}{\mathrm{i}}
\newcommand{\C}{\mathbb{C}}
\newcommand{\Z}{\mathbb{Z}}

\newcommand{\F}{\mathbb{F}}
\newcommand{\R}{\mathbb{R}}
\newcommand{\Q}{\mathbb{Q}}

\newcommand{\ini}{\mathrm{in}}
\newcommand{\adj}{\mathrm{adj}}

\DeclareMathOperator{\cl}{cl}
\DeclareMathOperator{\sgn}{sgn}

\usepackage[
text={440pt,575pt},
headheight=9pt,
centering
]{geometry}

\allowdisplaybreaks

\begin{document}

\date{September 2016}

\title{The geometry of rank-one tensor completion}

\author{Thomas Kahle}
\address{Otto-von-Guericke Universität Magdeburg\\ Magdeburg, Germany}
\email{thomas.kahle@ovgu.de}
\urladdr{\url{http://www.thomas-kahle.de}}

\author{Kaie Kubjas}
\address{Aalto University\\ Helsinki, Finland}
\email{kaie.kubjas@aalto.fi}
\urladdr{\url{http://www.kaiekubjas.com}}

\author{Mario Kummer}
\address{Max Planck Institute for Mathematics in the Sciences \\ Leipzig, Germany}
\email{kummer@mis.mpg.de}

\author{Zvi Rosen}
\address{University of Pennsylvania\\ Philadelphia, USA}
\email{zvihr@sas.upenn.edu}
\urladdr{\url{https://www.math.upenn.edu/~zvihr/}}

\makeatletter
  \@namedef{subjclassname@2010}{\textup{2010} Mathematics Subject Classification}
\makeatother

\begin{abstract}
The geometry of the set of restrictions of rank-one tensors to some of
their coordinates is studied.  This gives insight into the problem of
rank-one completion of partial tensors.  Particular emphasis is put on
the semialgebraic nature of the problem, which arises for real tensors
with constraints on the parameters.  The algebraic boundary of the
completable region is described for tensors parametrized by
probability distributions and where the number of observed entries
equals the number of parameters.  If the observations are on the
diagonal of a tensor of format $d\times\dots\times d$, the complete
semialgebraic description of the completable region is found.
\end{abstract}

\maketitle

\section{Introduction}\label{s:introduction}

When approaching high-dimensional tensor data, the large number of
entries demands complexity reduction of some sort. One important
structure to exploit is sparsity: tensors that have many zero entries
can be treated with specialized methods.  In this paper we focus on a
second concept, separability, which means that tensors have low rank
and thus can be parametrized by few parameters.  Specifically, we are
concerned with the \emph{rank-one completion problem}: Given a subset
of the entries of a tensor, does there exist a rank-one tensor whose
entries agree with the known data.

Tensor completion is a common task in many areas of science.  Examples
include compression problems~\cite{li2010tensor} as well as the
reconstruction of visual data~\cite{liu2013tensor} or
telecommunication signals~\cite{lim2014blind}.  Tensor completion also
appears in the guise of tensor factorization from incomplete data
which has many applications and
implementations~\cite{acar2011scalable}.  While we are specifically
concerned with the semialgebraic geometry of completability, most of
the literature deals with efficiency questions and approximate
solutions.  The main tool and mathematical hunting ground there is the
minimization of the nuclear
norm~\cite{gandy2011tensor,liu2013tensor,chandrasekaran2012convex,yu2014approximate,yang2015rank}.

Our approach here is directed towards the fundamental mathematical
question of a characterization of rank-one completability with a
particular emphasis on the real case.  What kind of constraints on the
entries of a partial tensor guarantee the completability to a rank-one
tensor?  As recognized in the matrix case, there are combinatorial
conditions on the locations of the known entries of a partial tensor.
In the best case, which here means working over an algebraically
closed field, all additional conditions are algebraic.  In all harder
and more interesting cases like the completion of real tensors,
tensors with linear constraints on the parameters, or even tensors
with inequality constraints on the parameters, the answer is almost
always semialgebraic, that is, it features inequalities.

The algebraic-combinatorial approach to matrix completion has come up
several times in the literature.  An original reference is the work of
Cohen et al.~\cite{cohen1989ranks}.  A more transparent proof of
necessary and sufficient conditions for the existence and uniqueness
of a rank-one completion was given in~\cite{hadwin2006rank}.  In this
work it became visible that there are combinatorial structures that
explain how completability depends on the locations of the observed
entries.  Algebraic and combinatorial structures underlying the
problem were further studied in
\cite{KiralyTheranTomioka,kiraly2013error}.  The semialgebraic nature
of low-rank completion problems is already visible in the matrix
case~\cite{KubjasRosen}.  The present paper continues and extends the
results of Kubjas and Rosen.
It has been recognized that solving tensor problems exactly is
systematically harder than solving matrix
problems~\cite{hillar2013most}.  In particular, low-rank tensor
completion is much more complicated than low-rank matrix completion
since tensor rank is much more complicated than matrix
rank~\cite{kolda2009tensor}.  The fact that tensor rank depends on the
field that one is working with also shines through here (see
Example~\ref{e:fieldDepend}).  Nevertheless, we conceive of our work
on the rank-one case as a stepping stone towards the low-rank case.

Our concrete approach is as follows: We consider the parametrization
map of rank-one tensors as tensor products of vectors.  Restricting
this map to a subset of the entries of the tensor, we get
parametrizations of partial tensors.  Recovery questions can then be
asked as questions about the images of the restrictions.  Over an
algebraically closed field, and with no further restrictions, the
image of the parametrization map is quite easy to describe.  This is
the classical Segre embedding from algebraic geometry.  For
applications, however, we need to work with constrained sets of
tensors and parameters.  For example, we may require that whenever the
observed entries of a partial tensor are real, the recovery ought to
be real too.  We may also choose to restrict parameters to be
nonnegative, sum to one, or both.  Examples~\ref{e:fieldDepend}
and~\ref{e:branching} show some immediate effects of these
requirements. The best possible result in our setup would be a method
that allows us to translate arbitrary inequalities and equations in
parameter space into a semialgebraic description of the image of
restrictions of~$\phi$.  We succeed with such a description in the
case of partial probability tensor (a tensor whose entries form a
probability distribution) with given diagonal entries
in~Section~\ref{s:diag-semialg-solution}.

We begin by illustrating the field dependence of tensor completion.  A
real tensor that has complex tensor rank one also has real tensor rank
one, but a similar statement for partial tensors is false.  The
principal problem is that a complex rank-one tensor can have some real
entries so that there exists no real rank-one tensor completing these
entries.  The following is an adaption of a standard example on the
difference between real and complex tensor rank, going back to
Kruskal~\cite{kruskal1983statement,kruskal1989rank}.
\begin{example}\label{e:fieldDepend}
Consider the $2\times 2\times 2$ partial tensor with third coordinate
slices
\begin{equation} \label{eq:partExample}
\begin{pmatrix}
? & 1 \\
1 & ? 
\end{pmatrix}\qquad\qquad
\begin{pmatrix}
1 & ? \\
? & -1
\end{pmatrix},
\end{equation}
where the $?$ stand for unspecified entries.
Proposition~\ref{p:tensorcomplete} below shows that the question marks
can be filled with complex numbers so that the resulting tensor has
rank one.  The question marks can, however, not be filled with real
numbers to make a real rank-one tensor.  Indeed, if this was the case,
then there would be real vectors $
\left(\begin{smallmatrix}
1\\
a
\end{smallmatrix}\right),
\left(\begin{smallmatrix}
1\\
b
\end{smallmatrix}\right),
\left(\begin{smallmatrix}
c\\
d
\end{smallmatrix}\right)
\in \R^2$ whose tensor product has the specified entries
in~\eqref{eq:partExample}. Here two entries can be chosen to be one by scaling
the first two vectors and compensating by the third.  In particular,
this means that
\[
bc = 1, \qquad ac = 1, \qquad d = 1, \qquad abd = -1.
\]
It is easy to check that there are only two complex solutions to these
equations.  In fact, the real rank-one completability of a partial
tensor like~\eqref{eq:partExample} does not depend on the exact values
of the specified entries, but only their signs.  The four entries can
be completed to a real rank-one tensor if and only if an even number
of them are negative.
\end{example}

The constraints in Example~\ref{e:fieldDepend} are given by equations.
The question becomes more interesting with semialgebraic constraints
as in the following example.
\begin{example}\label{e:branching}
Consider real rank-one $(2\times 2)$-matrices parametrized as
\[
\R^2 \times \R^2 \ni 
\begin{pmatrix}
\theta_{1,1}\\ \theta_{1,2}
\end{pmatrix},
\begin{pmatrix}
\theta_{2,1} \\ \theta_{2,2}
\end{pmatrix}
\mapsto
\begin{pmatrix}
\theta_{1,1} \theta_{2,1} & \theta_{1,1} \theta_{2,2} \\
\theta_{1,2} \theta_{2,1} & \theta_{1,2} \theta_{2,2}
\end{pmatrix}.
\]
Assume that only the diagonal entries can be observed.  It is easy to
see that in this case the set of possible diagonal entries is all
of~$\R^2$.  In applications in statistics,
$\theta_1 = (\theta_{1,1},\theta_{1,2})$ and
$\theta_2=(\theta_{2,1},\theta_{2,2})$ may be probability
distributions and satisfy $\theta_{1,2} = 1-\theta_{1,1}$ and
$\theta_{2,2} = 1-\theta_{2,1}$.  Not yet imposing nonnegativity, these
conditions constrain the diagonal entries $x_{11},x_{22}$ by
\[
(x_{11} - x_{22})^2 - 2(x_{11}+x_{22})  + 1 \geq 0.
\]
\begin{figure}[htpb]
    \centering
    \begin{subfigure}[b]{0.3\textwidth}
        \includegraphics[width=\textwidth]{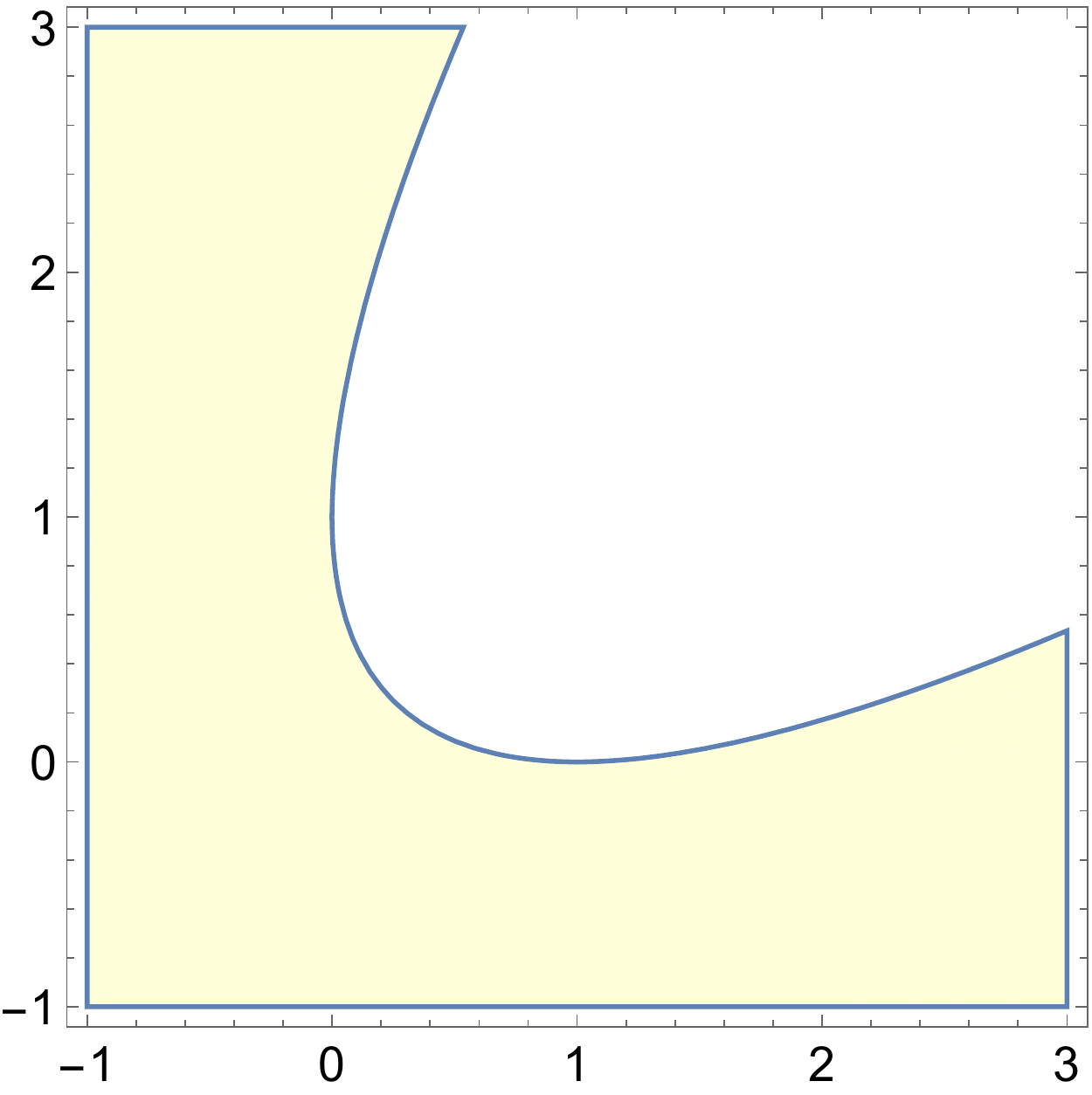}
        \caption{No inequalities}
        \label{fig:ex2A}
    \end{subfigure}
    ~ 
    \begin{subfigure}[b]{0.3\textwidth}
        \includegraphics[width=\textwidth]{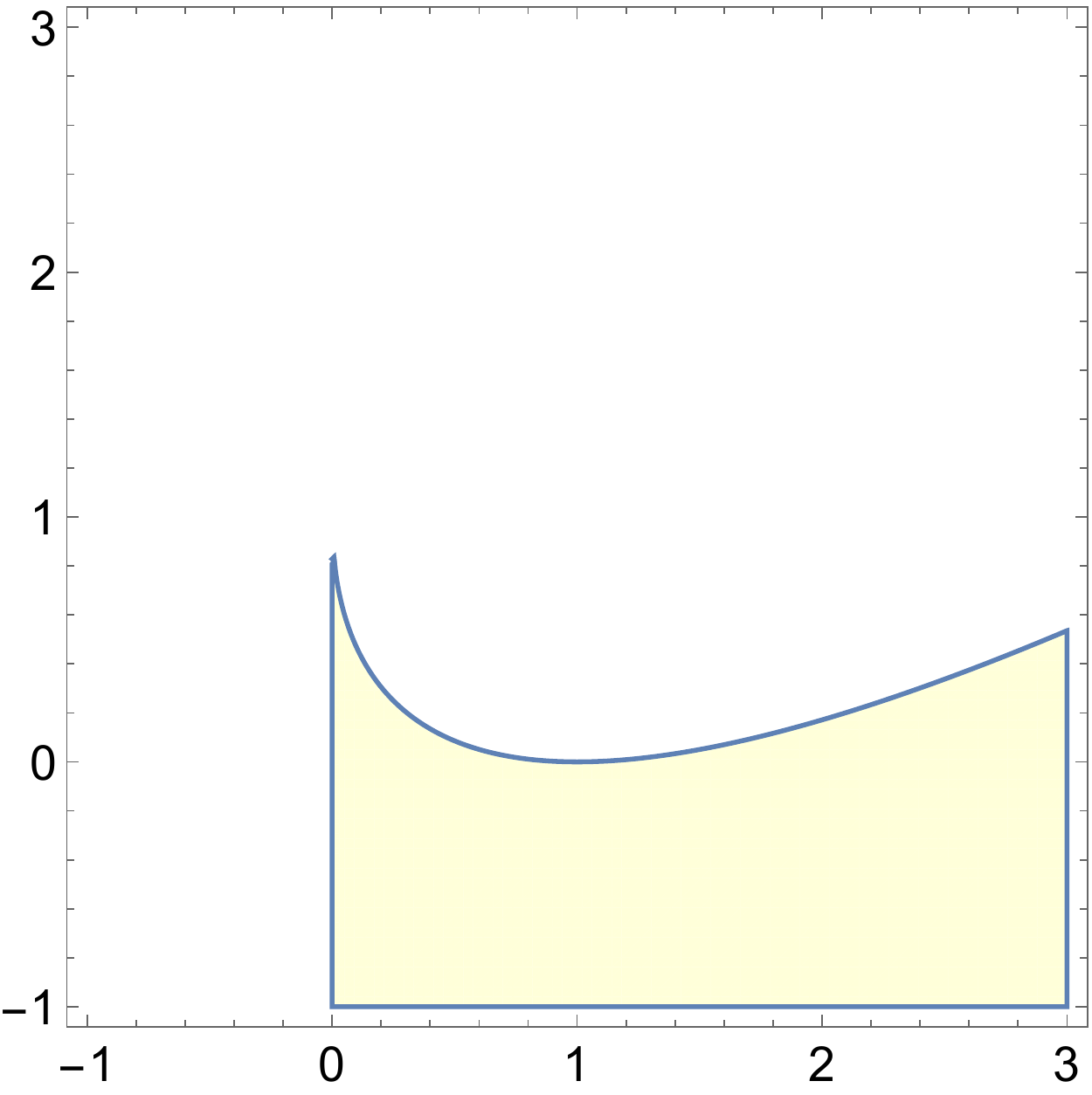}
        \caption{$\theta_{1,1},\theta_{2,1} \geq 0$}
        \label{fig:ex2B}
    \end{subfigure}
    ~ 
    \begin{subfigure}[b]{0.3\textwidth}
        \includegraphics[width=\textwidth]{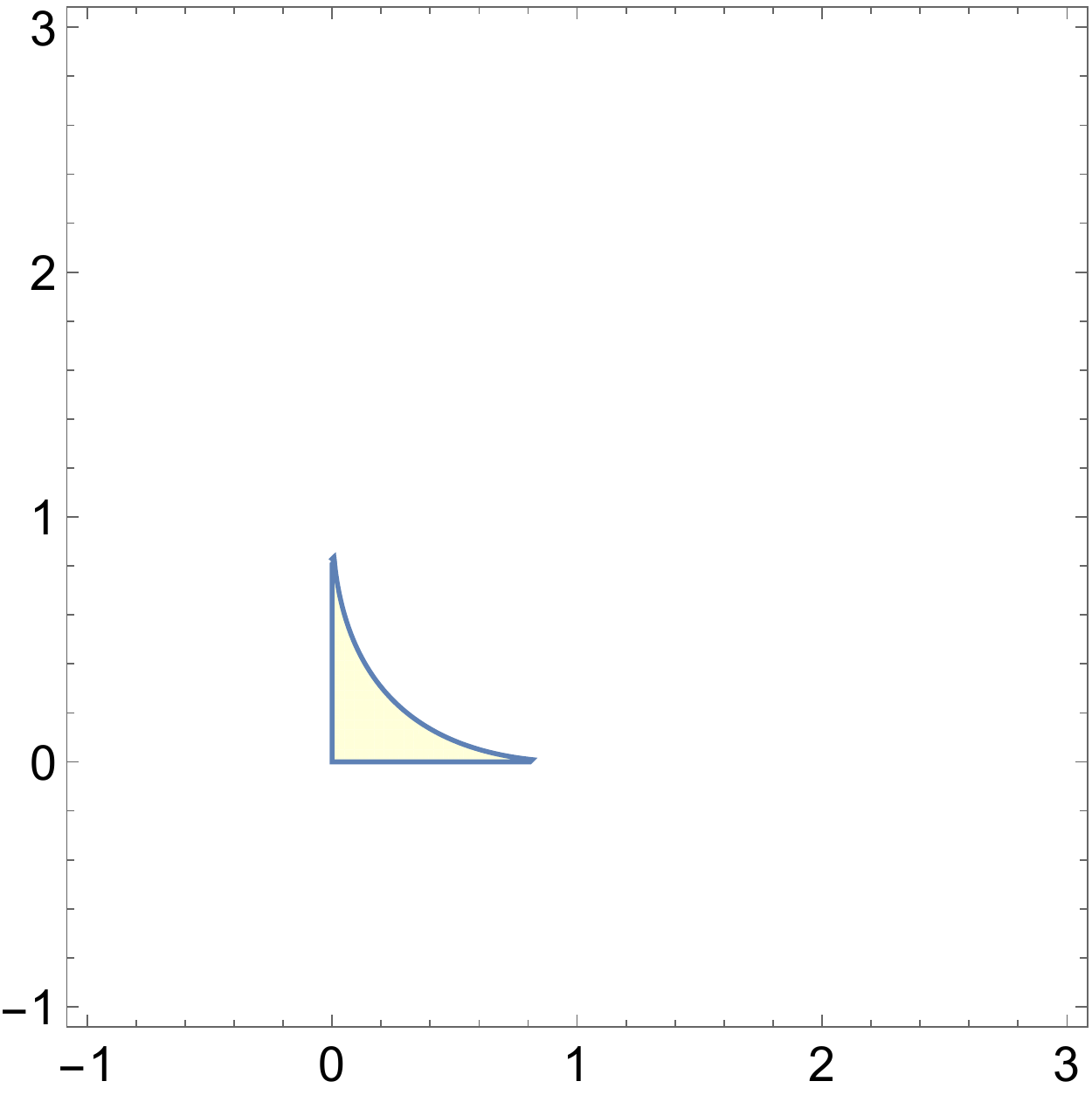}
        \caption{All parameters $\geq 0$}
        \label{fig:ex2C}
    \end{subfigure}
    \caption{Semialgebraic sets in Example~\ref{e:branching} \label{fig:ex2}}
\end{figure}
The yellow region in Figure~\ref{fig:ex2A} contains points satisfying
this constraint.  Introducing inequalities on the parameters
$\theta_1,\theta_2$ constrains the set of diagonal entries further.
For example, if $\theta_{1,1},\theta_{1,2},\theta_{2,1},\theta_{2,2}$
are the entries of probability vectors, they ought to be nonnegative.
Figure~\ref{fig:ex2B} shows the effect of imposing nonnegativity on
$\theta_{1,1}$ and $\theta_{2,1}$ (but not on $1-\theta_{1,1}$ and
$1-\theta_{2,1}$).  According to Theorem~\ref{t:diagIneqs}, imposing
also the last two conditions leads to the additional inequalities
$x_{11} \ge 0, x_{22} \ge 0$, and $x_{11} + x_{22} \le 1$
(Figure~\ref{fig:ex2C}).
\end{example}

In this paper the entries of tensors are indexed by
$D = [d_1]\times\dots\times[d_n]$ where $d_1,\dots,d_n$ are some fixed
integers, each larger than one.  A \emph{partial tensor} is an array
of real or complex numbers indexed by a subset $E\subseteq D$.  Field
assumptions are important in this work and we are precise about
whether we use $\R$ or~$\C$.

Let $\F$ be either $\R$ or $\C$.  The set of \emph{rank-one tensors}
in $\F^D$ is the image of the \emph{parametrization map}
\begin{equation}\label{eq:parametrize}
\F^{d_1} \times \cdots \times \F^{d_n} \to \F^D, \qquad (\theta_1,\dots,\theta_n) \mapsto
\theta_1\otimes \dots \otimes \theta_n.
\end{equation}
In classical algebraic geometry, the image is known as the Segre
variety.  It is characterized algebraically by quadratic binomials.
To see them, let $N_1\cup N_2 = [n]$ denote a partition of~$[n]$, and
$D_{i} = \prod_{k\in N_i} [d_k]$, $i=1,2$.  For each partition there
is a flattening of a tensor $T\in \F^D$ to a
matrix~$T \in \F^{D_1 \times D_2}$.  A tensor $T\in \F^D$ has rank at
most one if and only if all $(2\times 2)$-minors of all its
flattenings vanish.  This gives an explicit set of quadratic equations
in the indeterminates $x_i, i\in D$, representing the entries of a
tensor.  The equations can also be computed by implicitization of the
parametrization.  Consider the $\F$-algebra homomorphism
\begin{equation}\label{eq:segre-algebra-map}
\psi : \F[x_i : i \in D] \to \F[\theta_{j,k} : j\in[n], k \in[d_j]],\qquad
x_{i_1,\dots,i_n} \mapsto \prod_{j=1}^n \theta_{j,{i_j}}
\end{equation}
The toric ideal defining the Segre variety equals $\ker(\psi)$.  As
always in toric algebra, $\psi$ corresponds to a linear map
\begin{equation}\label{eq:segre-linear-map}
\Z^D \to \Z^{d_1 + \dots + d_n}, \qquad e_{i_1,\dots,i_n} \mapsto \sum_{j=1}^n e_{j,{i_j}},
\end{equation}
whose matrix in the standard basis we denote by
$A \in \{0,1\}^{(d_1+\dots +d_n)\times D}$.
See~\cite[Chapter~4]{sturmfels96GB} for an introduction to toric
algebra.  For any subset $E\subseteq D$, let $A_E$ be the matrix whose
columns are exactly the columns of $A$ with indices in~$E$.  The toric
ideal $I_E$ corresponding to $A_E$ equals the elimination ideal
$\ker(\psi) \cap \F[x_i:i\in E]$.  In general, equations for $I_E$ are
not easy to determine.  They could be easily read off a universal
Gröbner basis of $I_D$, but not much is known about this for general
$n$, even in the binary case $d_1=\dots=d_n=2$.
Isaac Burke has started to classify elements of the universal Gröbner
bases for binary rank-one tensors, but has encountered very intricate
combinatorial structures.  More information about this will appear in
his forthcoming PhD thesis.  If desired, the question of computing a
universal Gröbner basis can be formulated as a combinatorial question
about hypergraphs~\cite[Corollary~2.11]{PetrovicStasi}, but one should
not hope that the complications miraculously disappear in this
perspective.

In the remainder of the introduction we outline our specific results
together with concrete applications.  In Section~\ref{s:algComb} we
study the existence and finiteness of a rank-one completion of a
partial tensor when no further equations and inequalities are
assumed. We investigate the existence of completions over $\C$ and
$\R$ and in particular in the presence of zeros. Finitely completable
entries are characterized using matroid theory in
Proposition~\ref{p:clEexists}.  Unique completability of a partial
tensor is investigated in Corollary~\ref{c:unique_completability}.

Lower bounds on the number of observed entries for perfect recovery of
a low-rank matrix or tensor have been studied
in~\cite{candes2009exact,candes2010power,mu2014square,yuan2015tensor},
to name only a few references.  An important assumption in these
papers is that the positions of observed entries are sampled uniformly
randomly.  Proposition~\ref{p:clEexists} and
Corollary~\ref{c:unique_completability} characterize when a partial
tensor is finitely and uniquely completable to a rank-one tensor
without any assumptions on the sampling of the entries. This extends
the unique and finite completability in the matrix
case~\cite{hadwin2006rank,singer2010uniqueness,KiralyTheranTomioka} to
tensors. These results can also be used to design small sets of
locations of observed entries that guarantee finite or unique
completability to a rank-one tensor.  In the algebraic geometry
language above, tensor completion is equivalent to solving a system of
polynomial equations.  Thus Gr\"obner basis methods and numerical
algebraic geometry give effective methods for rank-one tensor
completion in the noiseless case, again with no assumptions on the
sampling of the entries.

Results in Section~\ref{s:algComb} also have possible applications in
chemometrics.  In~\cite{appellof1981strategies}, Appellof and Davidson
apply tensor decompositions of third order tensors to analysis of
multicomponent fluorescent mixtures.  There rank-one tensors
correspond to solutions with only one fluorophore.  In practice, some
of the measurements in a chemometry will not reach the excitation
level and are thus missing~\cite{andersen2003practical}.  Using our
results in Section~\ref{s:algComb}, one can verify whether tensors
with missing values correspond to solutions with only one fluorophore.
However, since real-world data has noise, in this and other
applications a further procedure to measure the distance from a
(semi)algebraic set would be necessary.

In Section~\ref{s:ramification} and in particular
Theorem~\ref{theorem:algebraic_boundary} we find a description of the
algebraic boundary of the completable region in the case that all
parameters belong to the probability simplex and the number of
observations equals the number of free parameters.  This is an
important step towards deriving a semialgebraic description of the
completable region.  Finally, Section~\ref{s:diag-semialg-solution}
illustrates in the diagonal case that totally effective semialgebraic
descriptions are possible to achieve.

\subsection*{Acknowledgments}
We thank Frank Sottile for suggesting Proposition~\ref{p:oddIndex}.
Part of this work was done while the authors visited the Simons
Institute for the Theory of Computing at UC Berkeley.  TK is supported
by the research focus dynamical systems of the state Saxony-Anhalt. MK
was supported by the Studienstiftung des deutschen Volkes. The visits
of MK and ZR to Aalto University were supported by the AScI thematic
program ``Challenges in large geometric structures and big data''

\section{Algebraic and combinatorial conditions for
completability}\label{s:algComb}

Throughout the section, let $E\subseteq D$ denote a fixed index set.
A \emph{partial tensor} is an element $T_E\in \F^E$.  Here the
subscript $E$ serves as a reminder that the tensor is partial.  For
each full tensor $T\in \F^D$, the \emph{restriction
$T_{\vert E} \in \F^E$ to $E$} consists of only those entries of $T$
indexed by $E$.  If a tensor
$T = \theta_1\otimes \dots \otimes \theta_n$ is of rank one, then any
zero coordinate in one of the $\theta_i$ yields an entire slice of
zeros in~$T$.  A first condition on completability of a partial tensor
results from this combinatorics of zeros.

\begin{definition}
Fix $j \in [n]$ and $i_j \in [d_j]$.  A \emph{maximal slice} of a
partial tensor $T\in\R^E$ is the tensor with index set
$E \cap [d_1]\times \dots \times [d_{j-1}] \times \{i_j\} \times
[d_{j+1}] \times \dots \times [d_n]$ which arises from $T$ by fixing
the $j$th index as~$i_j$.  A partial tensor is \emph{zero-consistent}
if every zero entry is contained in a maximal dimensional slice
consisting of only zero entries.
\end{definition}

The following proposition uses elimination so that $\F$ needs to be
algebraically closed.
\begin{proposition}\label{p:tensorcomplete}
A partial tensor $T_E \in \C^E$ equals the restriction of a rank-one
tensor $T\in\C^D$ to $E$ if and only if the following two conditions
hold
\begin{enumerate}
\item The partial tensor $T_E$ is zero-consistent.
\item\label{it:variety} The variety $V(I_E)$ contains~$T_E$.
\end{enumerate}
\end{proposition}
\begin{proof}
If $T\in \C^D$ is of rank one, then
$T = \theta_1\otimes \dots \otimes \theta_n$ for some vectors
$\theta_i\in\C^{d_i}$.  Therefore itself and any restriction to an
index set $E\subseteq D$ are zero-consistent.  It is also clear that
it lies in $V(I_D)$ and since $V(I_E)$ is the closure of the
projection of $V(I_D)$ it contains $T_{\vert E}$.

Now let $T_E \in \C^E$ be a zero-consistent partial tensor contained
in $V(I_E)$.  Without loss of generality, we can assume that $T_E$ has
no zero entry.  Indeed, from a partial tensor with consistent zeros we
can drop the zero-slices from the notation, complete, and then insert
appropriate zero slices to the completion.

For any $E' \subseteq D$, $I_{E'}$ is a toric ideal and has a
universal Gröbner basis consisting only of binomials $x^u - x^v$ such
that $x^u$ and $x^v$ are not divisible by a common variable.  Said
differently, when considered in a specific variable $x_i$, these
binomials are of the form $gx_i^n + h$ where $g,h$ are monomials that
do not use the variable~$x_i$. The extension
theorem~\cite[Theorem~3.1.3]{CLS96} states that outside the vanishing
locus of the polynomials~$g$, a point of $V(I_{E' \backslash \{i\}})$
can be extended to a point of~$V(I_{E'})$. Since all $g$ are
monomials, their vanishing locus is contained in the coordinate
hyperplanes.  Hence every partial tensor $T_E \in V(I_E)$ without zero
entries can be completed to $T \in V(I_D)$ by applying the extension
theorem $|D\setminus E|$ times.
\end{proof}

\begin{remark}\label{r:circuit_ideals}
The second condition in Theorem~\ref{p:tensorcomplete} need not
necessarily be checked on the toric ideal $I_{E}$ which may be
computationally unavailable.  There are several binomial ideals that
have $I_{E}$ as their radical and thus define the same variety.  A
natural example is the circuit ideal $C_E$ generated by all binomials
corresponding to circuits of $A_E$.  By
\cite[Proposition~8.7]{eisenbud96:_binom_ideal} it holds that
$V(C_E) = V(I_E)$.
\end{remark}

\begin{example}\label{e:realFalse}
Example~\ref{e:fieldDepend} shows that
Proposition~\ref{p:tensorcomplete} fails if $\C$ is replaced by~$\R$.
For $E=\{112,121,211,222\}$ the restricted matrix $A_E$ has full rank
and thus $V_\R(I_E) = \R^4$ while the given $T_E$ has no real rank-one
completion.
\end{example}

\begin{remark}\label{r:graphsForMatrices}
Rank-one matrix completion can be studied combinatorially using graph
theory.  Given a partial matrix $T_E \in \R^{E}$, one can define a
bipartite graph $G$ with vertex set $[d_1]\times [d_2]$ and edge
set~$E$.  The rank-one matrix completions are studied using
zero-entries and cycles of~$G$.  Zero-consistency of a matrix is
called \emph{singularity with respect to 3-lines}
in~\cite{cohen1989ranks} and {\em the zero row or column property}
in~\cite{hadwin2006rank}.  A partial matrix $T_E$ satisfies the second
condition in Proposition~\ref{p:tensorcomplete} if and only if on
every cycle $C$ of~$G$, the product over the edges with even indices
equals the product over the edges with odd indices.  This observation
follows from the explicit description of the generators of the
universal Gröbner basis in terms of cycles on the bipartite
graph~\cite[Proposition~4.2]{villarreal1995rees}. In
\cite{Ohsugi_Hibi_1999} the Graver basis is computed, which in this
case coincides with the universal Gröbner basis.  This condition is
called \emph{singularity with respect to cycles}
in~\cite{cohen1989ranks} and {\em the cycle property}
in~\cite{hadwin2006rank}. Rank-one matrices have a square-free
universal Gröbner basis. Therefore all uses of the extension theorem
as in Proposition~\ref{p:tensorcomplete} yield linear equations.  In
the tensor case the combinatorial interpretations break down and, for
example, the universal Gröbner basis of rank-one $2\times 2\times 2$
tensors is of degree three and not square-free.  However, iterative
algorithms using the extension theorem or related methods also work
for tensors.
\end{remark}

The problem of real rank-one completion is, for each $E\subseteq D$,
to determine the difference of the image of real and complex rank-one
tensors $T$ under the restriction map $T\mapsto T_{\vert E}$.  This
problem leads to a combinatorial problem in $\Z$-linear algebra
clarified in Proposition~\ref{p:oddIndex}.  Consider a partial tensor
$T_E\in\R^E$ that satisfies the conditions in
Proposition~\ref{p:tensorcomplete}.  As in the proof of
Proposition~\ref{p:tensorcomplete} we can assume that $T_E$ has no
zero entries as these would be contained in zero slices which we can
ignore for completion.  The parametrization \eqref{eq:parametrize} of
the entries $T_e, e\in E$ of a completion of $T_E$ takes the form
\begin{equation}\label{e:laurent}
T_e = \theta_{1,e_1}\cdots \theta_{n,e_n}, \qquad e\in E,
\end{equation}
where $\theta_{j,e_j}$ denotes the $e_j$th component of~$\theta_j$.
Completability questions are questions about the solutions of this
system of binomial equations.  Additionally assume that $E$ meets
every maximal dimensional slice of~$D$, which implies that every
parameter $\theta_{j,k}$ appears at least once in the
equations~\eqref{e:laurent}.  Given $T_e \neq 0, e\in E$, this implies
that any solution has only nonzero values for the parameters.  This
means that the ideal generated by~\eqref{e:laurent} can be considered
as an ideal in Laurent polynomial ring
$\F[\theta_{j,k}^{\pm}, j=1,\dots,n, k\in[d_j]]$, and the theory of
\cite[Section~2]{eisenbud96:_binom_ideal} applies.  In particular, the
equations can be diagonalized by computing the Smith normal form of
the exponent vectors of the monomials appearing in~\eqref{e:laurent},
which corresponds to a multiplicative coordinate change.

\begin{example}\label{e:laurentDiag}
The equations in Example~\ref{e:fieldDepend} can be diagonalized to
\[
x^2 = -1, y = 1, z = 1, w = 1,
\]
where $x^2 = ab(bc)^{-1}(ac)^{-1} = c^{-2}$ , $y=ac$, $z=bc$, $w=d$.
\end{example}

As a consequence of the diagonalization argument that
gives~\cite[Theorem~2.1(b)]{eisenbud96:_binom_ideal}, we get the
following proposition.  

\begin{proposition}\label{p:oddIndex}
For a given subset $E\subseteq D$ the following are equivalent:
\begin{enumerate}[(i)]
\item Every real partial tensor $T_E \in \R^E$ with nonzero entries
which is completable over the complex numbers is also completable over
the real numbers.
\item The index of the lattice spanned by the columns of $A_E$ in its
saturation is odd.
\end{enumerate}
Moreover, whether a real partial tensor $T_E \in \R^E$ with nonzero
entries which is completable over the complex numbers is also
completable over the real numbers depends only on the signs of the
entries of~$T_E$.
\end{proposition}

\begin{proof}
Since we assume the entries to be nonzero $(i)$ is equivalent to the
homomorphism of tori $\psi: (\C^*)^r\to(\C^*)^s$ (for suitable $r$ and
$s$) corresponding to our parametrization having the property that
every real point in the image has a real point in its preimage.  After
applying suitable group automorphisms of $(\C^*)^r$ and $(\C^*)^s$ we
can assume that $\psi$ is of the form
$\psi(x_1,\ldots,x_r)=(x_{1}^{a_1},\ldots,x_r^{a_r},1,\ldots,1)$ if
$s \geq r$ and $\psi(x_1,\ldots,x_r)=(x_{1}^{a_1},\ldots,x_s^{a_s})$
otherwise (this corresponds to computation of the Smith normal form
of~$A_E$).  Group automorphisms of $(\C^*)^r$ send real points to real
points.  Thus $(i)$ holds if and only if all $a_i$ are odd. Since the
index of the lattice spanned by the columns of $A_E$ in its saturation
is the product of the $a_i$ the claim follows.

For the second part assume that $a_1,\dots,a_k$ are odd and
$a_{k+1},\ldots,a_r$ are even.  Consider first the case that
$s \geq r$.  Then $T_E$ is completable over the real numbers if and
only if it is in the preimage of
\[(\R^*)^k\times (\R_{>0})^{r-k}\times \{1\}^{s-r}\]
under the above automorphism of $(\C^*)^s$. Assuming completability
over the complex numbers, this translates to conditions on the signs
of the entries of~$T_E$.  A similar argument applies for $s < r$.
\end{proof}

In Proposition~\ref{p:oddIndex}, the assumption that the partial
tensor $T_E$ has nonzero entries is no loss of generality, since any
zero entry of a rank-one tensor is contained in a maximal dimensional
slice of zeros.  These maximal dimensional slices of zeros originate
from parameters being zero and can be dealt with separately.

Given completability over a fixed field, one can ask about uniqueness
properties of the completion.  More generally, for some of the
completed entries there could be only finitely many choices while
others can vary continuously.  An entry that has only finitely many
possible values for rank-one completion, is called \emph{finitely
completable}. An entry that has only one possible value for rank-one
completion, is called \emph{uniquely completable}. The occurrence of
finitely completable entries is natural.  For example, if three
entries of a rank-one $2\times 2$ matrix are given, the determinant
becomes a linear polynomial determining the fourth entry.  The proof
of Proposition~\ref{p:tensorcomplete} shows how the finitely
completable entries are solutions of certain binomial equations.  In
this context, an important observation is that, generically, the
locations of the finitely completable entries do not depend on the
entries of $T_E$, but just the combinatorics of~$E$. This statement
for matrices can be found in~\cite[Theorem 10]{KiralyTheranTomioka}.

\begin{proposition}\label{p:clEexists}
There is a matroid on ground set $D$ with closure function
$\cl : 2^D \to 2^D$ having the following property: Let $E\subseteq D$
be any index set with $T_E \in \C^E$ a generic partial tensor
completable according to Proposition~\ref{p:tensorcomplete}.  Then the
closure $\cl(E)$ consists exactly of the entries that are finitely
completable from the entries in~$E$.
\end{proposition}
\begin{proof}
This follows from Proposition~\ref{p:closureExplicit} below.
\end{proof}

The gist of Proposition~\ref{p:clEexists} is that for generic $T_E$,
the set of finitely completable entries does not depend on the entries
of $T_E$, but only on~$E$.  Even more, $\cl(E)$ is an honest closure
relation on explicit matroids.  The following matroids can be used.
\begin{itemize}
\item The column matroid of the Jacobian of the
parametrization~\eqref{eq:parametrize}.
\item The algebraic matroid of the toric ideal $\ker(\psi)$
in~\eqref{eq:segre-algebra-map}.
\item The column matroid of the matrix $A$
defining~\eqref{eq:segre-linear-map}.
\end{itemize}

The equivalence of these three matroids is well-known.  The algebraic
matroid of the coordinate ring of a toric ideal equals the linear
matroid of the defining matrix.  The equivalence of the first and
second matroid follows from \cite[Proposition~2.14]{KRT13}.  The
closure function can be specified algebraically as follows.  For any
index set $E\subseteq D$, let $\C[E] := \C[x_e : e\in E]$ be a
polynomial ring with one indeterminate for each entry of a partial
tensor with index set~$E$.

\begin{proposition}\label{p:closureExplicit}
The closure function $\cl : 2^D \to 2^D$ of the algebraic matroid
defined by the ideal $I$ is the function which maps a set
$E \subseteq D$ to the largest set $E'$ containing $E$ such that the
projection $\C^{E'} \to \C^{E}$ induces a generically finite-to-one
map on $V(I_D \cap \C[E']) \to V(I_D \cap \C[E])$.
\end{proposition}
\begin{proof}
The closure $\cl(E)$ of $E$ consists of all elements dependent on $E$.
For an algebraic matroid, $\cl(E)$ consists of all elements algebraic
over~$E$.  Given a point in $V(I_D \cap \C[E])$ and
$f\in\cl(E)\setminus E$, substituting the indeterminates $x_e, e\in E$
with its coordinates, the algebraic dependence of $f$ on $E$ yields a
univariate polynomial in~$x_f$.  As for the elements not in the
closure of $E$: since they are not algebraic over~$E$, they can take
infinitely many values; in particular, the fiber over the generic
point in that projection is infinite.
\end{proof}

Proposition~\ref{p:oddIndex} also gives a characterization when a
partial tensor is uniquely completable to a rank-one tensor.  Also
here the assumption on nonzero entries can be dealt with separately.

\begin{corollary}\label{c:unique_completability}
\begin{enumerate}[(i)]
\item A partial tensor with nonzero entries is uniquely completable to
a complex rank-one tensor if and only if it is finitely completable
and the lattice spanned by the columns of $A_E$ is saturated.
\item A real partial tensor with nonzero entries is uniquely
completable to a real rank-one tensor if and only if it is finitely
completable and the index of the lattice spanned by the columns of
$A_E$ in its saturation is odd.
\end{enumerate}
\end{corollary}

\begin{proof}
As in the proof of Proposition~\ref{p:oddIndex}, we can assume that
$\psi$ is of the form
$\psi(x_1,\ldots,x_r)=(x_1^{a_1},\ldots,x_r^{a_r},1,\ldots,1)$. A
point in the image has a unique complex preimage if and only if
$a_1,\ldots,a_r$ are all one. A real point in the image has a unique
real preimage if and only if $a_1,\ldots,a_r$ are all odd.
\end{proof}

\begin{example}\label{e:finiteComplete}
In the matrix case, the finitely completable entries of a generic
partial matrix form a block diagonal partial matrix after a suitable
indexing of rows and columns (we do not assume that matrices or blocks
are square matrices).  The reason is that the closure operation on the
Jacobian matroid can be interpreted as the closure operation on the
graphic matroid of the bipartite graph whose vertices are the row and
column labels and whose edges correspond to~$E$.  This closure
completes connected components to complete bipartite graphs, and a
bipartite graph where all connected components are complete
corresponds to a block diagonal matrix after a suitable indexing of
rows and columns.  An analogous statement for tensors is not true.
For example, the partial $2 \times 2 \times 2$ tensor with observed
entries at positions $(1,1,2),(1,2,1),(2,1,1)$ (the blue entries in
Figure~\ref{fig:222flats}) has no finitely completable entries and it
cannot be transformed to block diagonal form after a suitable
indexing.
\begin{figure}[htpb]
\begin{center}
\includegraphics[scale=1.2]{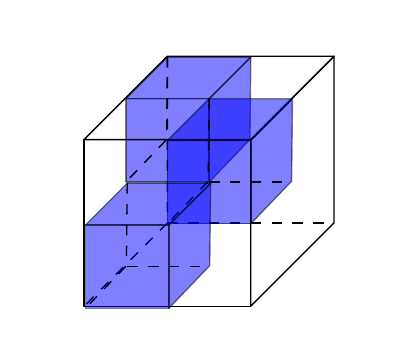}
\caption{$2 \times 2 \times 2$ tensor with observed entries
$(1,1,2),(1,2,1),(2,1,1)$}
\label{fig:222flats}
\end{center}
\end{figure}
\end{example}

\section{The algebraic boundary of the completable region}
\label{s:ramification}

The \textit{algebraic boundary} of a semialgebraic set
$S\subseteq \R^n$ is the Zariski closure of the boundary of $S$ in the
Euclidean topology.  Theorem~\ref{theorem:algebraic_boundary} is a
description of the algebraic boundary of the completable region in the
case that the parameters form probability distributions and the number
of observations equals the number of free parameters.  A result of
Sinn gives that the algebraic boundary is defined by a single
polynomial (Proposition~\ref{prop:pure_codim_one}) which we show to be
a product of a special irreducible polynomial with indeterminates.  To
this end we study the Jacobian of the parametrization map and the
factorization of its determinant.  We compute the locus
where the Jacobian has rank deficit (Propositions~\ref{p:JacobianForm}
and~\ref{prop:linear_form}), and then argue about the relation of this
set to the algebraic boundary.

The set $E \subseteq D$ again denotes the index set of observed
entries of a tensor.  Consider the restricted parametrization
\[
p:\Delta^{d_1-1} \times \cdots \times \Delta^{d_n-1} \rightarrow \R^E,
\] 
where $\Delta^{m-1}$ is the $(m-1)$-dimensional simplex that is the
convex hull of the unit vectors in~$\R^m$.  We write
$N = \{(j,k) : j \in [n], k \in [d_j]\}$ for the set indexing
dimensions of the general parametrization~\eqref{eq:parametrize} of
rank-one tensors.  In the above parametrization the parameters are
linearly dependent and this dependence, together with nonnegativity
leads to semialgebraic constraints on the image of $p$, the
\emph{completable region}.  Write
$\tilde{N} = \{(j,k) : j \in [n], k \in [d_j-1]\}$ for the index set
of linearly independent parameters. From now on, let $E \subseteq D$
be of size $\sum_{i=1}^n(d_i-1)$. With this requirement, we set the
number of observations equal to the dimension of the parameter space.

The vanishing ideal of the graph of the parametrization is
$G_E=\<x_i-p_{i}: i \in E\>$.  By assumption, the Jacobian matrix $J_E$
of $p$ is a square matrix.  We also assume that the completable region
has nonempty interior and that for each pair $(j,k) \in N$ there is at
least one element in $E$ that has $k$ at its $j$th position.  If it
has empty interior, then its algebraic boundary is just its Zariski
closure, which can be determined by eliminating the parameters from
the vanishing ideal $G_E$ of the graph of~$p$.  This happens for example
when the number of observations exceeds the dimension of the parameter
space, that is $|E| > \sum_{i=1}^n (d_i-1)$.  The assumption on the
index set $E$ is necessary so that the map $p$ captures information
about each parameter.  It is satisfied exactly if $E$ meets every
maximal dimensional slice of~$D$.  For example, for matrices this
means that $E$ meets each row and each column.

First we show that the algebraic boundary of the completable region is
defined by a single nonzero polynomial.  The following lemma is a
version of the generic smoothness lemma.  In its proof we use the same
notation as~\cite{Hart77}.
\begin{lemma}\label{l:genericSmoothness}
Let $g:\R^k\to\R^k$ be a polynomial map whose image has nonempty
interior. The Jacobian determinant is not identically zero.
\end{lemma}

\begin{proof}
Let $h:\mathbb{A}^k\to\mathbb{A}^k$ be the morphism of affine
varieties given by the same polynomials as~$g$. By assumption $h$ is
dominant. Thus, by \cite[Lem. III.10.5]{Hart77} there is a nonempty
Zariski open subset $U\subseteq\mathbb{A}^k$ such that $h|_U$ is
smooth (of relative dimension zero). Thus, by
Proposition~\cite[Prop. III.10.4]{Hart77} the sheaf of relative
differentials of $U$ over $\mathbb{A}^k$ is locally free which means
that the Jacobian matrix is invertible at all points of~$U$.
\end{proof}

\begin{lemma}\label{l:polynomialImage}
Let $g: \R^k\to\R^k$ be a polynomial map.  Let $S\subseteq \R^k$ be a
semialgebraic set contained in the closure of its interior in the
Euclidean topology. If the image $g(S)$ of $S$ has nonempty interior,
then $g(S)$ is contained in the closure of its interior in the
Euclidean topology.
\end{lemma}

\begin{proof}
Let $J$ be the Jacobian matrix of $g$. Since $g(S)$ has nonempty
interior, the Jacobian determinant $\det(J)$ is not identically zero
by Lemma~\ref{l:genericSmoothness}.  Since
$D=\{x\in\R^k: \det(J(x))=0\}$ contains no nonempty open set and by
the assumption on~$S$, the closure of
$S'=\textrm{int}(S) \smallsetminus D$ in the Euclidean topology
contains~$S$. It follows that the closure of $g(S')$ in the Euclidean
topology contains~$g(S)$. The inverse function theorem implies that
$g(S')$ is contained in the interior of~$g(S)$. Thus, the claim
follows.
\end{proof}

\begin{proposition}\label{prop:pure_codim_one}
The algebraic boundary of the completable region is of pure
codimension one, that is, it is the zero set of a nonzero polynomial.
\end{proposition}

\begin{proof}
By \cite[Lemma 4.2]{sinn13}, if a semialgebraic set $S \subset \R^k$
is nonempty and contained in the closure of its interior in the
Euclidean topology and the same is true for its complement
$\R^k \backslash S$, then the algebraic boundary of $S$ is a variety
of pure codimension one. We will show that these assumptions hold for
the image of $p$ and its complement. The image of $p$ is clearly
nonempty. It is contained in the closure of its interior in the
Euclidean topology by Lemma~\ref{l:polynomialImage}. Also, the image
of $p$ is clearly not all of $\R^m$ since each coordinate takes a
value between $0$ and $1$ and it is closed in the Euclidean topology
since $p$ is continuous and maps from a compact space. Thus, the
complement of the image of $p$ is nonempty and open in the Euclidean
topology. Therefore, the assumptions of \cite[Lemma 4.2]{sinn13} are
satisfied and the claim follows.
\end{proof}

\subsection{The Jacobian determinant of the parametrization}
\label{s:jacobidet}
In order to find the polynomial that defines the algebraic boundary of
the completable region we compute the determinant of the
Jacobian~$J_E$ (Theorem~\ref{theorem:algebraic_boundary}).  The
following example illustrates the results in the next two subsections.

\begin{example}\label{example:bit_matrix}
Assume the indices of observed entries of a $2\times 2\times 2$ tensor
are $(2,1,1),(1,2,1),(1,1,2)$.  Denote $l_i=1-\theta_i$ for
$i=1,2,3$. Define the ideal
\[
G_E=\langle x_{211} - l_1\theta_2 \theta_3, x_{121} -
\theta_1 l_2 \theta_3,x_{112} - \theta_1 \theta_2
l_3 \rangle.
\]
The Jacobian matrix of the parametrization map equals
\[
J_E=
\begin{pmatrix}
-\theta_2 \theta_3 & l_1\theta_3 & l_1\theta_2\\
l_2\theta_3 & -\theta_1 \theta_3 & \theta_1 l_2\\
\theta_2 l_3 & \theta_1l_3 & -\theta_1 \theta_2
\end{pmatrix}
\]
and has determinant
\[
\theta_1^2 \theta_2 \theta_3 + \theta_1 \theta_2^2 \theta_3 + \theta_1 \theta_2 \theta_3^2 - 2 \theta_1 \theta_2 \theta_3= \theta_1 \theta_2 \theta_3 (-\theta_1 - \theta_2 - \theta_3 +2).
\]
The product of polynomials of the parametrization map is
$\theta_1^2 \theta_2^2 \theta_3^2 l_1 l_2 l_3$.  Division by
$\theta_1 \theta_2 \theta_3 l_1 l_2 l_3$ yields the monomial
$m = \theta_1\theta_2\theta_3$ which divides the determinant as
claimed by Proposition~\ref{p:JacobianForm}. Consider the matrix
\begin{equation*}
B_E=
\begin{pmatrix}
0 & 1 & 1 & 1\\
1 & 0 & 1 & 1\\
1 & 1 & 0 & 1
\end{pmatrix}.
\end{equation*}
The construction of this matrix is explained in
Section~\ref{subsection:linear_factor}.  The kernel of $B_E$ is
spanned by $v = (-1,-1,-1,2)^T$.  Let
$l_E=-\theta_1-\theta_2-\theta_3 + 2$ be the linear polynomial whose
coefficients are equal to the entries of~$v$, as suggested by
Proposition~\ref{prop:linear_form}.  The Jacobian determinant
equals~$l_E$ multiplied with~$m$.
\end{example}

Compared to the general parametrization~\eqref{eq:parametrize}, the
map $p$ has linear restrictions on the coordinates of its domain.  We
prove our results for a slight generalization of this.  To this end,
let $\theta_{j,k}, (j,k) \in \tilde{N}$ be indeterminates.  For any
$(j,k) \in N$, let $l_{j,k}$ be a linear polynomial in the
indeterminates~$\theta_{j,1},\dots,\theta_{j,d_j-1}$.  In our context
 $l_{j,k} = \theta_{j,k}$ for
$(j,k)\in\tilde{N}$ and
$l_{j,d_j} = 1 - \sum_{k=1}^{d_j-1}\theta_{j,k}$ otherwise. In this representation, the parametrization takes
the form
\begin{equation}\label{eq:parametrizeLinearPoly}
p_{i_1,\dots,i_n} = \prod_{j=1}^n l_{j,i_j}.
\end{equation}
We prove the following results for this generalized parametrization.
An entry of the Jacobian matrix is indexed by a pair
$((i_1,\dots,i_n),(j,k))$ of indices $(i_1,\dots,i_n) \in E$,
$(j,k) \in N$ and is given by
\[
\frac{\partial}{\partial \theta_{j,k}} \prod_{m=1}^n l_{m,i_m} =
\frac{\partial l_{j, i_j}}{\partial \theta_{j,k}} \prod_{\substack{m=1\\m\neq j}}^n l_{m, i_m}.
\]
In the Leibniz determinant formula, each summand is a product
\begin{equation}\label{eq:det_explicit}
\sum_{\sigma \in S_{|E|}} \sgn(\sigma)
\prod_{(j,k)\in N}
\frac{\partial l_{j, \sigma(i)_j}}{\partial
\theta_{j,k}}\prod_{\substack{m=1\\m\neq j}}^n
l_{m, \sigma(i)_m}.
\end{equation}
For $j\in [n]$ and $k\in [d_j]$, let
$\alpha(j,k) = |\{i\in E: i_j = k \}|$ denote the number of times, the
linear polynomials $l_{j,k}$ is used in the parametrizations~$p_i$ for
$i\in E$.  Each term, and thus the entire determinant
in~\eqref{eq:det_explicit} is divisible by the product of linear
polynomials
\begin{equation}\label{eq:prod_lin_factor}
\prod_{(j,k) \in N} l_{j,k}^{\alpha(j,k) - 1}.
\end{equation}
Consider the polynomial $l_E$ arising upon division of the Jacobian
determinant by the product~\eqref{eq:prod_lin_factor}.  We show a
degree bound on the determinant which yields that $l_E$ is either zero
or a polynomial of degree at most one.  We use some technical lemmata,
the first of which is inspired by
\cite[Lemma~4.7]{plaumann2013determinantal}.
\begin{lemma}\label{l:detDivide}
Let $\F$ be a field.
Let $M$ be an $n\times n$ matrix with entries in
$\F[x_1,\dots,x_m]$.  If an irreducible polynomial
$f \in \F[x_1,\dots,x_m]$ divides every $(r+1)$-minor of $M$, then
$f^{n-r}$ divides~$\det(M)$.
\end{lemma}
\begin{proof}
Without loss of generality we can assume $\det(M)\neq 0$.  The proof
is by induction on~$n$.  For $n=1$ the statement is trivial.  If $n$
is arbitrary and $r\ge n-1$, then the assumption and the conclusion
are the same.  If $r < n-1$, the induction hypothesis yields that
$f^{n-r-1}$ divides every $(n-1)$-minor of~$M$.  Therefore the
adjugate matrix $\adj(M)$ can be factored as $\adj(M) = f^{n-r-1}M'$
for some $n\times n$ matrix~$M'$ with entries in $\F[x_1,\dots,x_m]$.
It follows that
\[
\det(M)^{n-2}M = \adj(\adj(M)) = \adj(f^{n-r-1}M') = f^{(n-r-1)(n-1)}\adj(M').
\]
If $f$ divides every entry of $M$, then $f^n$ divides $\det(M)$.  If
it does not, then $f^{(n-r-1)(n-1)}$ divides $\det(M)^{n-2}$.  This
implies that $f^{n-r}$ divides $\det(M)$: If $s$ is the power
with which $f$ appears in the factorization of $\det(M)$, then
$s\le n-r-1$ implies $s(n-2) \le (n-r-1)(n-2) < (n-r-1)(n-1)$.
\end{proof}

\begin{lemma}\label{l:det_degree}
Let $\F$ be a field and $M$ an $n\times n$ matrix whose entries
$M_{ij} \in \F[x_1,\dots,x_m]$ are not necessarily homogeneous
polynomials of degree at most $d$.  Let $\ini(M)$ be the matrix whose
$(i,j)$th entry is the standard graded leading form of the $(i,j)$th
entry of $M$.  Let $Q$ be the quotient field of $\F[x_1,\dots,x_m]$
and let $r$ be the dimension of the kernel of $\ini(M)$ considered as
a $Q$-linear map. Then $\deg(\det(M)) \le nd - r$.
\end{lemma}
\begin{proof}
Let $M^h$ denote the matrix whose entries are the homogenizations of the entries of $M$ by a new
indeterminate~$x_0$ so that $\ini(M) = M^h_{|x_0 = 0}$.  We show that
$\det(M^h)$ is divisible by $x_0^r$.  By
assumption, all $(n-r+1)$-minors of $\ini(M)$ vanish and thus each
$(n-r+1)$-minor of $M^h$ is divisible by~$x_0$. It follows from
Lemma~\ref{l:detDivide} that $\det(M^h)$ is divisible by~$x_0^r$. Since the degree of $\det(M^h)$ is at most $nd$ and since $M=M^h_{|x_0 = 1}$, the claim follows.
\end{proof}

\begin{lemma}\label{l:detDegreeBound}
For any $E\subseteq D$ with $|E| = |\tilde{N}|$, the Jacobian
determinant~\eqref{eq:det_explicit} has degree at most $(|\tilde{N}|-1)(n-1)$.
\end{lemma}
\begin{proof}
Let $\ini(p_i)$ be the leading form of~$p_i$. If $\theta_{j,k}$
appears in $p_i$, then it appears in $\ini(p_i)$ with degree one.
Hence the matrix of leading forms of the Jacobian is
$\ini(J) = \left(\frac{\partial \ini(p_i)}{\partial
  \theta_{j,k}}\right)_{i,(j,k)}$.  By Lemma~\ref{l:det_degree} it
suffices to exhibit $(n-1)$ linearly independent relations among the
columns of $\ini(J)$.
As visible from~\eqref{eq:parametrizeLinearPoly}, for each
$j \in [n]$, the polynomial $\ini(p_i)$ is a homogeneous function in
the subset $\Theta_j = \{\theta_{j,k}, k\in[d_j]\}$ of the variables.
Therefore, according to Euler's theorem on homogeneous functions,
\[
\sum_{k=1}^{d_j} \theta_{j,k} \frac{\partial \ini(p_i)}{\partial
\theta_{j,k}} = \ini (p_i), \qquad j \in [n], i\in E.
\]
Equating these for different $j$ yields $n-1$ relations among the
columns of~$\ini(J)$
\[
\sum_{k=1}^{d_j} \theta_{j,k} \ini(J)_{i,(j,k)} =
 \sum_{k=1}^{d_{j'}} \theta_{j',k} \ini(J)_{i,(j',k)}, \qquad j,j' \in [n], i\in E. \qedhere
\]
\end{proof}

The degree bound in Lemma~\ref{l:detDegreeBound} together with the
divisibility by the product~\eqref{eq:prod_lin_factor} implies the
following form of the Jacobian determinant.
\begin{proposition}\label{p:JacobianForm}
The determinant of the Jacobian equals 
\[
l_E \prod_{(j,k) \in N} l_{j,k}^{\alpha(j,k) - 1},
\] 
where $l_E \in \R[\theta_{j,k}, (j,k)\in D]$ is of degree at most one.
\end{proposition}

\begin{proof}
By Lemma~\ref{l:detDegreeBound} the degree of $l_E$ is bounded from
above by
\[
(|\tilde{N}|-1)(n-1)-\sum_{(j,k) \in N} (\alpha(j,k) - 1)=(|\tilde{N}|-1)(n-1)-n|E|+|N|.
\]
Since $|E| = |\tilde{N}|=|N|-n$, this bound equals
\[
(|E|-1)(n-1)-n|E|+|E|+n=1.\qedhere
\]
\end{proof}

\subsection{A linear factor of the Jacobian determinant}
\label{subsection:linear_factor}
To determine the linear form $l_E$ in Proposition~\ref{p:JacobianForm}
we restrict back to the most relevant case in which the linear forms
$l_{(j,k)}$ are equal to $\theta_{(j,k)}$ whenever
$(j,k) \in \tilde{N}$ and $1-\sum_k\theta_{(j,k)}$ otherwise.
Consider the matrix $A$ defining the linear
map~\eqref{eq:segre-linear-map}.  From $A$, extract the submatrix
$\tilde{B}_E$ consisting of the rows corresponding to indices
$(j,k)\in\tilde{N}$ and the columns corresponding to indices in~$E$
and let $\tilde{B}_E^T$ be its transpose.  For the entries
$b_{i,(j,k)}$ of $\tilde{B}_E^T$ this means that $b_{i,(j,k)} = 1$ if
$i$ has $k$ at its $j$th position and $b_{i,(j,k)}=0$ otherwise.  In
other words, the $(i,(j,k))$ entry of $\tilde{B}_E^T$ is $1$ if and
only if the parameter corresponding to $(j,k)$ appears in~$p_i$.  Let
\[
B_E=
\begin{pmatrix} 
& 1  \\
\tilde{B}_E^T & \vdots \\
& 1
\end{pmatrix}.
\]
Under the assumption that the completable region has non-empty
interior, $B_E$ has full rank.  Indeed, its transpose defines the
toric ideal~$I_E$ which is zero by the assumption.  In this section we
show how the kernel of $B_E$ determines the linear polynomial~$l_E$.
To this end, consider the matrix $\tilde{J}_E$ that arises from $J_E$
by dividing the $i$th row by $p_i$ for each $i \in E$, and multiplying
the column $(j,k)$ by $\theta_{k,l}$ for each $(j,k) \in
\tilde{N}$. The following lemma is immediate from
Proposition~\ref{p:JacobianForm}.
\begin{lemma}\label{l:tildeJdet}
 The determinant of $\tilde{J}_E$ equals
 \[
  \frac{l_E}{\prod_{j=1}^n (1-(\theta_{j,1}+\ldots+\theta_{j,d_j-1}))}.
 \]
\end{lemma}

\begin{lemma}\label{l:BEdet}
The determinant of $\tilde{B}_E^T$ is the constant term of $l_E$.
\end{lemma}

\begin{proof}
The matrix $\tilde{B}_E^T$ arises from $\tilde{J}_E$ by evaluating all
indeterminates $\theta_{(j,k)}$ at zero.  By Lemma~\ref{l:tildeJdet}
its determinant is the constant of~$l_E$.
\end{proof}

\begin{lemma}\label{l:linFormCoeff}
The coefficient of $\theta_{(j,k)}$ in $l_E$ is the determinant of the
matrix arising from $\tilde{B}_E^T$ after replacement of the
$(j,k)$-column with the all $-1$ vector.
\end{lemma}

\begin{proof}
Let $C_1$ be the matrix obtained from multiplying the $(j,k)$-column
of $\tilde{J}_E$ with $(1-(\theta_{j,1}+\ldots+\theta_{j,d_j-1}))$.
Let $C_2$ be the matrix obtained from $C_1$ after evaluating all
parameters except for $\theta_{j,k}$ at zero.  By construction, the
determinant of $C_2$ is $a+b \theta_{j,k}$ where $a$ is the constant
term of $l_E$ and $b$ is the coefficient of $\theta_{j,k}$ in
$l_E$. Thus, the matrix $C_3$ obtained from $C_2$ by evaluating
$\theta_{j,k}$ at $1$ has determinant $a+b$.  By construction, the
$i$th entry of the $(j,k)$th column of $C_3$ is $-1$ if
$(1-(\theta_{j,1}+\ldots+\theta_{j, d_j-1}))$ appears in $p_i$ and
zero otherwise.  All other columns of $C_3$ are equal to the
respective columns of~$\tilde{B}_E^T$.  Let $C_4$ be the matrix
obtained from $C_3$ after subtracting the columns corresponding to
$(j,k')$ with $k'\neq k$ from the column $(j,k)$.  The determinant of
$C_4$ is $a+b$ and the $i$th entry of the $(j,k)$th column of $C_4$ is
$0$ if $\theta_{j,k}$ appears in $p_i$ and $-1$ otherwise.  Let $C_5$
be the matrix obtained from $C_4$ after subtracting from the $(j,k)$th
column the $(j,k)$th column of $\tilde{B}_E^T$.  By multilinearity of
the determinant and Lemma~\ref{l:BEdet}, $\det(C_5)=b$ and $C_5$ is
precisely the matrix obtained from $\tilde{B}_E^T$ after replacing the
column corresponding to $(j,k)$ with the all $-1$ vector.
\end{proof}

\begin{proposition}\label{prop:linear_form}
The kernel of $B_E$ is one-dimensional.  Let $v$ be in the kernel of
$B_E$ and let $l_v$ be the linear polynomial whose constant term is
the last entry of $v$ and whose coefficient of $\theta_{j,k}$ is the
$(j,k)$th entry of $v$. Then $l_v$ is $l_E$ multiplied with a scalar.
\end{proposition}

\begin{proof}
The vector $w$ of signed maximal minors of $B_E$ is in the kernel of
$B_E$ by Laplace expansion. By Lemmas~\ref{l:BEdet}
and~\ref{l:linFormCoeff}, $l_E$ is the linear polynomial whose
constant term is the last entry of $w$ and whose coefficient of
$\theta_{j,k}$ is the $(j,k)$th entry of $w$. Since $l_E$ is not zero,
not every maximal minor of $B_E$ vanishes. Thus, its kernel is
one-dimensional and every other element of the kernel is a scalar
multiple of~$w$.
\end{proof}

\subsection{Computing the algebraic boundary from the Jacobian determinant}
According to Proposition~\ref{p:JacobianForm}, the Jacobian
determinant $\det (J_E)$ is the product of (\ref{eq:prod_lin_factor})
and the polynomial $l_E$.  Since we assume that the completable region
has nonempty interior, we get $l_E \neq 0$. The main result of this
section is the following.
\begin{theorem}\label{theorem:algebraic_boundary}
Eliminating the parameter variables from the ideal $G_E+\<l_E\>$ where
$G_E$ is the vanishing ideal of the graph of $p$ gives an ideal
generated by a non-constant irreducible polynomial $f$. The polynomial
$q$ that defines the algebraic boundary of the completable region
is the product of $f$ with some coordinates.
\end{theorem}

\begin{proof}
If a point on the boundary of the completable region is the image of a
point of the interior of the parameter space, the inverse function
theorem implies that the Jacobian determinant vanishes at this
point. Since the image of $p$ is closed in the Euclidean topology, its
boundary is contained in the image of the union of the boundary of the
parameter space with the zero set of the Jacobian determinant. By the
assumptions on~$E$, the boundary of the parameter space is mapped to a
subset of the union of the coordinate hyperplanes. The same holds for
the components of the zero set of the Jacobian determinant
corresponding to the factors of~\eqref{eq:prod_lin_factor}. Thus the
algebraic boundary is contained in the union of the coordinate
hyperplanes and the image of the zero set of $l_E$ (which is
irreducible). It cannot be contained in the union of the coordinate
hyperplanes, since the image has nonempty intersection with the
positive quadrant and is bounded by the hyperplane of entries summing
to one. Furthermore, since the algebraic boundary is of pure
codimension one by Proposition~\ref{prop:pure_codim_one}, the claim
follows.
\end{proof}

\begin{remark}
One could ask for a variant of
Theorem~\ref{theorem:algebraic_boundary} in which nonnegativity is not
imposed on the domain of $p$.  In this case the proof of
Proposition~\ref{prop:pure_codim_one} fails since the image of $p$
need not be closed.  One could consider the same problem under the
additional assumption that the map $p$ is proper.  Then the proofs
show that the algebraic boundary is still given as the zero set of a
factor of the product of $f$ with the coordinates (the possibility of
the factor being 1 is not excluded).
\end{remark}

In Section~\ref{s:diag-semialg-solution} we find the complete
semialgebraic description of tensors of format
$d\times \dots \times d$ with observed diagonal entries.  For general
$E$ this may be too hard, but it would be interesting to understand
the behavior of the degree of the boundary hypersurface.

\begin{problem}\label{prob:degree}
Determine the degree of the irreducible polynomial $f$ in
Theorem~\ref{theorem:algebraic_boundary} as a function of
$n,d_1,\dots,d_n$, and~$E$.
\end{problem}

\begin{example}\label{e:boundary}
We continue Example~\ref{example:bit_matrix}.
\begin{equation*}
G_E+\langle l_E \rangle = \langle x_{211} - l_1\theta_2 \theta_3, x_{121}
- \theta_1 l_2 \theta_3,x_{112} - \theta_1 \theta_2 l_3, -\theta_1 - \theta_2 - \theta_3 +2\rangle.
\end{equation*}
Eliminating $\theta_1, \theta_2$, and $\theta_3$ yields a prime ideal
generated by
\begin{multline*}
x_{211}^4x_{121}^2-2x_{211}^3x_{121}^3+x_{211}^2x_{121}^4-2x_{211}^4x_{121}x_{112}+2x_{211}
     ^3x_{121}^2x_{112}+2x_{211}^2x_{121}^3x_{112}\\
-2x_{211}x_{121}^4x_{112}+x_{211}^4x_{112}^2+2x_{211}
     ^3x_{121}x_{112}^2-6x_{211}^2x_{121}^2x_{112}^2+2x_{211}x_{121}^3x_{112}^2+x_{121}^4x_{112}^2\\
-2
     x_{211}^3x_{112}^3+2x_{211}^2x_{121}x_{112}^3+2x_{211}x_{121}^2x_{112}^3-2x_{121}^3x_{112}^3+
     x_{211}^2x_{112}^4-2x_{211}x_{121}x_{112}^4\\
+x_{121}^2x_{112}^4-2x_{211}^3x_{121}^2-2x_{211}^2
     x_{121}^3+8x_{211}^3x_{121}x_{112}-4x_{211}^2x_{121}^2x_{112}+8x_{211}x_{121}^3x_{112}\\
-2x_{211}^3
     x_{112}^2-4x_{211}^2x_{121}x_{112}^2-4x_{211}x_{121}^2x_{112}^2-2x_{121}^3x_{112}^2-2x_{211}^2
     x_{112}^3+8x_{211}x_{121}x_{112}^3\\
-2x_{121}^2x_{112}^3+x_{211}^2x_{121}^2-10x_{211}^2x_{121}x_{112}
     -10x_{211}x_{121}^2x_{112}+x_{211}^2x_{112}^2-10x_{211}x_{121}x_{112}^2\\
+x_{121}^2x_{112}^2+4x_{211}
     x_{121}x_{112}.
\end{multline*}
The zero set of this polynomial together with the coordinate
hyperplanes is the algebraic boundary of the set of
$2 \times 2 \times 2$ partial tensor with specified entries at
positions $(2,1,1),(1,2,1)$ and $(1,1,2)$ which can be completed to a
rank-one tensor inside the standard simplex, see
Figure~\ref{fig:exampleAlgebraicBoundary}.
\begin{figure}[htpb]
\centering
\includegraphics[width=0.4\textwidth]{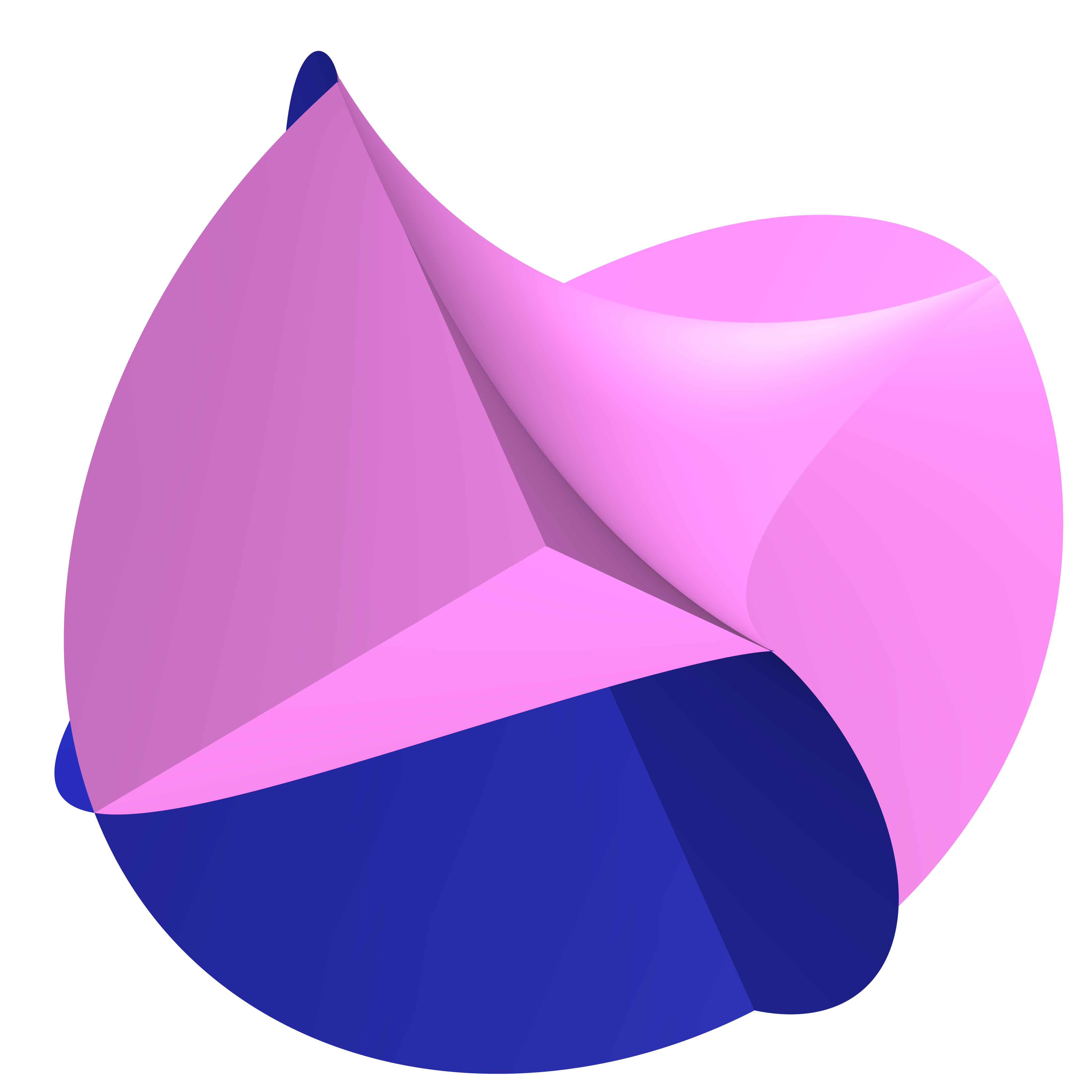}
\caption{The irreducible surface that is part of the boundary of the
completable region in Example~\ref{e:boundary}.  The completable
triples inside $[0,1]^3$ reside below the bent triangular shape.  The
surface is singular along the coordinate axes and the sides of the
bent triangle.  The algebraic boundary of the completable region
also includes the coordinate hyperplanes since the bounded region
below the bent triangle also extends into negative coordinates.}
\label{fig:exampleAlgebraicBoundary}
\end{figure}
\end{example}

Next to elimination, Sturm sequences provide another method to
retrieve information about the algebraic boundary.  They work directly
with the image coordinates and could yield lower complexity algorithms
to produce the boundary of the completable region.  The nature of the
construction of Sturm's sequence warrants hope that this would yield
some control over the degree in Problem~\ref{prob:degree}.  We present
an example illustrating the method.

\begin{example}\label{e:sturm-boundary}
As in Examples~\ref{example:bit_matrix} and~\ref{e:boundary}, consider
$2 \times 2 \times 2$ partial tensors with three observed entries
$x_{112},x_{121}, x_{211}$.  As argued in
Example~\ref{e:finiteComplete}, the entry at the position $(1,1,1)$ is
not finitely completable.  Let $x$ be an indeterminate standing for
this entry.  After picking $x$, the remaining values of the tensor all
satisfy algebraic equations in the given entries and~$x$.  The slices
of the tensor $T$ then are
\[
\begin{pmatrix}
x & x_{121} \\[3mm]
x_{211} & \frac{x_{121}x_{211}}{x}
\end{pmatrix}
\begin{pmatrix}
x_{112} &  \frac{x_{112}x_{121}}{x} \\[3mm]
\frac{x_{112}x_{211}}{x} & \frac{x_{112}x_{121}x_{211}}{x^2}
\end{pmatrix}.
\]
If $T$ is a probability tensor, then its entries should sum to one.
  Let
$e_i,$ be the $i$th elementary symmetric function on the letters
$x_{112},x_{121},x_{211}$. This leads
to the following constraint on~$x$
\begin{equation*}
x + e_1 + \frac{e_2}{x} + \frac{e_3}{x^2} - 1=0.
\end{equation*}
To find conditions which guarantee the existence of real or positive
solutions $x$ we examine the Sturm sequence of this constraint after
clearing denominators.  The first three polynomials in the Sturm
sequence are
\begin{gather*}
f_0(x) = \theta(x) = x^{3} + (e_1 - 1) x^{2} + e_2 x + e_3,\\
f_1(x) = \theta'(x) = 3 x^{2} + 2 (e_1 -1) x + e_2,\\
f_2(x) = \frac{2}{9}\left(e_{1}^2 - 3e_2 - 2e_1 + 1 \right) x + \frac{1}{9} \left(e_{2}e_{1} - e_2  - 9 e_3 \right).
\end{gather*}
The constant $f_3 = -\text{rem}(f_1,f_2)$ in the Sturm sequence is a
longish quotient of two polynomials in the elementary symmetric
polynomials~$e_1,e_2,e_3$.  We omit printing it here, since it can be
reproduced easily with computer algebra.  To apply Sturm's
theorem~\cite[Theorem~1.4]{sturmfels02:_solvin_system_polyn_equat}, we
evaluate at $x=0,1$.  Assuming $e_1 \leq 1$,
\[
f_0(0) = e_3 \geq 0,\quad f_1(0) = e_2 \geq 0,\quad f_2(0) = -e_2(1 -
e_1) - 9e_3 \leq 0.
\]
Let $\sigma$ be the sign of the constant~$f_3$.  At $x=1$ we find
\[
f_0(1)=e_1+e_2+e_3 \geq 0 \quad \text{and} \quad f_1(1)=1+2e_1+e_2 \geq 0.
\]
Denote the sign of $ f_2(1)=2e_1^2 - 4e_1-7e_2+2+e_1 e_2 - 9e_3$ by
$\mu$.  Assuming that $x_{112},x_{121},x_{211}$ are in the interior of
$\Delta^2$, the sign sequence at zero is $++-\ \sigma$ and at one is
$++\mu\ \sigma$.  According to Sturm's theorem $f_0(x)$ has a root in
the half-open interval $(0,1]$ if and only if $\mu = +$ and
$\sigma=+$.  Hence the completable region in the interior of
$\Delta^2$ is defined by
$x_{112}> 0, x_{121} > 0, x_{211} > 0, 1- e_1 > 0,f_2(1) \geq 0$ and
$f_3 \geq 0$.  By Theorem~\ref{theorem:algebraic_boundary}, a single
irreducible polynomial in the $x_e, e\in E$ together with coordinate
hyperplanes gives the algebraic boundary.  Explicit computation shows
that the numerator of $f_3$ equals a scalar multiple of the generator
of the ideal in Example~\ref{e:boundary}.
\end{example}

\section{Completability of diagonal partial probability tensors}
\label{s:diag-semialg-solution}

We give a semialgebraic description of the region of diagonal partial
tensors that can be completed to rank-one probability tensors.  The
following theorem is our starting point and appeared
as~\cite[Proposition~5.2]{KubjasRosen}.

\begin{theorem}
Let
$E=\{(1,\ldots,1),(2,\ldots,2),\ldots,(d,\ldots,d)\} \subseteq [d]^n$.
A diagonal partial tensor $T_E \in \R^E_{\geq 0}$ is completable to a
rank-one tensor in $\Delta^{d^n-1}$ if and only if
\[
\sum_{i=1}^d x_{i,\dots,i}^{\frac{1}{n}} \leq 1.
\]
\end{theorem}
Denote
 \[
S_{n,d}=\{x \in \R_{\geq 0}^d: \sum_{i=1}^d x_{i,\dots,i}^{\frac{1}{n}} \leq 1 \}.
\]
It was shown in~\cite{KubjasRosen} that $S_{n,d}$ is a semialgebraic
set and a description of its algebraic boundary was given. We show for
any integers $n,d \geq 1$ that the set is a basic closed semialgebraic
set and we construct the defining polynomial inequalities.  We prepare
some lemmata about real zeros of polynomials $f\in\R[t]$.  To this end
let $f^{(i)}$ denote the $i$th derivative of~$f$.

\begin{lemma}\label{lemma:sa1}
Let $f \in \R[t]$ be a monic polynomial of degree~$d$.  Let
$\epsilon \in \R$ such that $f^{(i)}(\epsilon) \geq 0$ for all
$i=0,\ldots,d-1$. Then $\epsilon \geq \alpha$ for every real zero
$\alpha \in \R$ of $f$.
\end{lemma}

\begin{proof}
The statement is true when the number of real zeros (counted with
multiplicity) of $f$ is at most one because $f$ is monic.  We proceed
by induction on the degree~$d$.  By the above observation the case
$d\leq 1$ is clear.  Let $d\geq 2$ and let $f$ have $e \geq 2$ real
zeros $\alpha_1 \leq \ldots \leq \alpha_e$ (counted with
multiplicity).  If $\alpha_{e-1}=\alpha_{e}$ is a double root of~$f$,
it is also a root of $f'$ and by the induction hypothesis,
$\epsilon \geq \alpha_e$.  If $\alpha_{e-1} < \alpha_{e}$, then by
Rolle's theorem there is a $\beta \in \R$ with
$\alpha_{e-1} < \beta < \alpha_e$ and $f'(\beta)=0$.  Thus by
induction hypothesis, $\epsilon \geq \beta$.  Since $\alpha_e$ is a
simple root, $f$ has a change of signs at $\alpha_e$.  Since $f$ is
monic, it is negative between $\beta$ and $\alpha_e$, and thus
$\epsilon \geq \alpha_e$.
\end{proof}

\begin{lemma}\label{l:connI}
 Let $f \in \R[t]$ be a monic polynomial of degree $d$.
 The set 
 \[
  I=\{\epsilon \in \R: f^{(i)}(\epsilon) \geq 0 \textrm{ for all } i=0,\ldots,d-1 \}
 \]
 is connected and thus a closed, unbounded interval. 
\end{lemma}

\begin{proof}
Assume that there are real numbers $a < b < c$ such that $a,c \in I$
but $b \not\in I$.  There is an $1 \leq i <d$ such that
$f^{(i)}(b)<0$. Since $f^{(i)}(a)$ and $f^{(i)}(c)$ are nonnegative,
by Rolle's theorem and the intermediate value theorem, there is a
$\xi>a$ with $f^{(i+1)}(\xi)=0$. This contradicts $a \in I$ by
Lemma~\ref{lemma:sa1} applied to $f^{(i+1)}$.

The interval $I$ is closed because a finite number of polynomials
being nonnegative is a closed condition. It is unbounded because the
defining polynomials are monic and thus nonnegative for sufficiently
large $t$.
\end{proof}

\begin{lemma}\label{l:largerThanRealPart}
If a polynomial $f$ has a real zero $\alpha \in \R$ that is larger
than the real part of any other zero of~$f$, then
\[
\{\epsilon \in \R: f^{(i)}(\epsilon) \geq 0 \textrm{ for all }
i=0,\ldots,d-1 \}=\{\epsilon \in \R: \epsilon \geq \alpha\}.
\]
\end{lemma}

\begin{proof}
Consider the factorization of $f$ as
\[
f=\prod_{k=1}^s (t-(a_k+b_k\i))(t-(a_k-b_k\i)) \prod_{l=1}^r (t-c_l)
\] 
with $2s+r=d$ and real numbers $a_k,b_k,c_l \in \R$ that satisfy $a_k,c_l \leq \alpha$.
Then
\[
f=\prod_{k=1}^s ((t-\alpha)^2+2 (\alpha - a_k) (t-\alpha)
+(\alpha-a_k)^2+b_k^2) \prod_{l=1}^r ((t-\alpha)+(\alpha-c_l)).
\]
Thus, as a polynomial in $t-\alpha$, $f$ has nonnegative coefficients.
This shows $f^{(i)}(\alpha) \geq 0$ for all $i=0,\ldots,d-1$, from
which the statement follows by Lemmas~\ref{lemma:sa1}
and~\ref{l:connI}.
\end{proof}

Let $d,n \geq 1$ be integers. For every tuple
$\sigma \in \{0,\ldots, n-1\}^d$ we define the linear polynomial
$x_\sigma:=\sum_{i=1}^d \zeta_n^{\sigma_i} x_i \in
L[x_1,\ldots,x_d]$ where $\zeta_n \in \C$ is a primitive $n$th root of
unity and $L=\Q[\zeta_n]$.  Now consider the polynomial
\[Q_{n,d}=\prod_{\sigma \in \{0,\ldots, n-1\}^d} (t-x_\sigma) \in
L[t,x_1,\ldots,x_d].\]
Since $Q_{n,d}$ is fixed under the action of the Galois group of $L$
over~$\Q$, it has rational coefficients.  Since $Q_{n,d}$ is stable
under scaling $t$ or one of the $x_i$ by an $n$th root of unity, there
exists a polynomial $\tilde{Q}_{n,d} \in \Q[t,x_1,\ldots,x_d]$ of
degree $n^{d-1}$ with
$Q_{n,d}(t,x_1,\dots,x_d)=\tilde{Q}_{n,d}(t^n,x_1^n,\ldots,x_d^n)$.
For $i=0,\ldots,n^{d-1}-1$ let
\[P_{n,d,i}=\frac{\partial^i \tilde{Q}_{n,d}}{\partial t^i} |_{t=1}
\in \Q[x_1,\ldots,x_d]\]
be the $i$th derivative of $\tilde{Q}_{n,d}$ evaluated at $t=1$.

\begin{example}
Let $d=n=2$. We have 
\begin{align*} Q_{2,2} & = (t-x_1-x_2)
(t-x_1+x_2) (t+x_1-x_2) (t+x_1+x_2) \\
& = t^4-2t^2x_1^2-2t^2x_2^2+x_1^4-2x_1^2x_2^2+x_2^4.                   
\end{align*}
As predicted, $Q_{2,2}$ is a polynomial in $t^2,x_1^2,x_2^2$. We have
$Q_{2,2}(t,x_1,x_2)=\tilde{Q}_{2,2}(t^2,x_1^2,x_2^2)$ with
$\tilde{Q}_{2,2}=t^2-2tx_1-2tx_2+x_1^2-2x_1x_2+x_2^2$.
\end{example}

\begin{theorem}\label{t:diagIneqs}
A nonnegative vector $x \in \R_{\geq 0}^d$ is an element of $S_{n,d}$
if and only if $P_{n,d,i}(x) \geq 0$ for all $0 \leq i < n^{d-1}$.  If
$n$ is odd, then
$S_{n,d}=\{x \in \R_{\geq 0}^d: P_{n,d,0}(x) \geq 0 \}$.
\end{theorem}

\begin{proof}
Fix $z \in \R_{\geq 0}^d$. The roots of
$\tilde{Q}_{n,d}(t,z_1,\ldots,z_d) \in \R[t]$ are precisely the
complex numbers $(\sum_{i=1}^d \zeta_n^{\sigma_i} \sqrt[n]{z_i})^n$
for $\sigma \in \{0,\ldots,n-1\}^d$.  Indeed, the roots of $Q_{n,d}$
are the numbers $(\sum_{i=1}^d \zeta_n^{\sigma_i} z_i)$.  Since
$\tilde{Q}_{n,d}(t^n,z_1,...,z_d)=Q_{n,d}(t,\sqrt[n]{z_1},\ldots,\sqrt[n]{z_d})$,
the zeros of $\tilde{Q}(t^n,z_1,\ldots,z_d)$ are
$(\sum_{i=1}^d \zeta_n^{\sigma_i} \sqrt[n]{z_i})$,
i.e.~$\tilde{Q}((\sum_{i=1}^d \zeta_n^{\sigma_i} \sqrt[n]{z_i})^n
,z_1,\ldots,z_d)=0$.  The real zero
$\alpha=(\sum_{i=1}^d \sqrt[n]{z_i})^n\in \R$ is larger than the real
part of any other zero.  By Lemma~\ref{l:largerThanRealPart}, for
every $\epsilon \in \R$,
\[
\epsilon \geq \alpha \Leftrightarrow \frac{\partial^i
\tilde{Q}_{n,d}}{\partial t^i}(\epsilon,z_1,\ldots,z_d) \geq 0 \: \text{ for all
} i=0,\ldots,n^{d-1}-1.
\]
With $\epsilon=1$, this gives the first part of the claim.  If $n$ is
odd, then $\alpha$ is the only real zero of
$\tilde{Q}_{n,d}(t,z_1,\ldots,z_d) \in \R[t]$ and thus
$\epsilon \geq \alpha$ if and only if
$\tilde{Q}_{n,d}(\epsilon,z_1,\ldots,z_d) \geq 0$.
\end{proof}

Let $e_{i,d}$ denote the $i$th elementary symmetric polynomial in
$x_1, \ldots, x_d$.

\begin{example}
 Let $d=n=2$. Then we have \[\tilde{Q}_{2,2}=t^2-2 t (x_1+x_2)+(x_1-x_2)^2.\]
 Thus, $S_{2,2}$ is defined by the following inequalities:
 \begin{align*}
  x_1,x_2 &\geq  0 \\
  1-x_1-x_2 & \geq  0 \\
  1-2 (x_1+x_2)+(x_1-x_2)^2 &\geq  0,
 \end{align*}
 which can be rewritten as
  \begin{align*}
  x_1,x_2 &\geq  0 \\
  1-e_{1,2} &\geq  0 \\
  (1-e_{1,2})^2-4 e_{2,2} & \geq 0 .
 \end{align*}
\end{example}

\bibliography{biblio}
\bibliographystyle{plain}

\end{document}